\documentclass[11pt, amsfonts]{amsart}

\usepackage{amsmath,amssymb,amsthm,color,enumerate,comment,bm}
\usepackage[all]{xy}
\usepackage{tikz}
\usepackage{todonotes}
\usepackage{url} 
\usepackage{hyperref} 

\numberwithin{equation}{section}

\SelectTips{eu}{12}

\newtheorem{thm}{Theorem}[section]

\newtheorem{theorem}[thm]{Theorem}

\newtheorem{lemma}[thm]{Lemma}

\theoremstyle{definition}

\newtheorem{definition}[thm]{Definition}

\newtheorem{remark}[thm]{Remark}

\theoremstyle{remark}

\newcommand{\ZZ}{\mathbb{Z}}

\newcommand{\RR}{\mathbb{R}}

\newcommand{\hatR}{\widehat{R}}

\newcommand{\Cay}{\mathrm{Cay}}
\newcommand{\Bor}{\mathrm{Bor}}
\newcommand{\length}{\mathrm{length}}
\newcommand{\Ker}{\mathrm{Ker}}
\newcommand{\ind}{\mathrm{ind}}

\newcommand{\NNN}{\mathcal{N}}
\newcommand{\BBB}{\mathcal{B}}

\newcommand{\sgn}{\mathrm{sgn}}
\newcommand{\hPsi}{\hat{\Psi}}
\newcommand{\hx}{\hat{x}}

\date{}

\title[Cayley graphs and circular chromatic numbers]{Cayley graphs of abelian groups, circular colorings and discrete fundamental groups}

\author[K. Enami]{Kengo Enami}
\address{College of Liberal Arts and Sciences,  Kitasato University, 1-15-1 Kitasato, Minami-ku, Sagamihara, Kanagawa 252-0373, Japan}
\email{enamikengo@gmail.com}

\author[T. Matsushita]{Takahiro Matsushita}
\address{Department of Mathematical Sciences, Faculty of Science, Shinshu University, Matsumoto, Nagano 390-8621, Japan}
\email{matsushita@shinshu-u.ac.jp}

\subjclass[2020]{Primary 05C15
}

\keywords{circular chromatic number; Cayley graphs; neighborhood complex; $k$-fundamental group}

\begin{document}

\baselineskip.525cm

\begin{abstract}
For each integer $k \ge 2$, we prove that an abelian group $\Gamma$ satisfies $\exp(\Gamma) \in \{1,\dots,k\} \cup \{2,4,\dots,2k\}$ if and only if it admits no non-bipartite Cayley graph whose circular chromatic number is strictly less than $2 + \frac{2}{k-1}$. This refines the algebraic characterization by Krebs and Sankar of abelian groups that admit $3$-chromatic Cayley graphs. 
\end{abstract}

\maketitle

\section{Introduction}\label{section introduction}

The chromatic numbers of Cayley graphs have attracted considerable attention \cite{Alon,Babai,BNT,Czerwinski,CK,DHY,GK,GKM,GR2,Payan}, in part because of their connections with information theory \cite{AO}. Krebs and Sankar \cite{KS} recently provided an algebraic characterization of the abelian groups that admit a $3$-chromatic Cayley graph. The principal purpose of this paper is to refine their characterization in terms of circular chromatic numbers.

Throughout this paper, graphs may be infinite and may have self-loops, but they have no multiple edges. For a symmetric subset $S$ of a group $\Gamma$, the \emph{Cayley graph $\Cay(\Gamma, S)$} is the graph defined as follows: The vertex set of $\Cay(\Gamma, S)$ is $\Gamma$, and two elements $x$ and $y$ in $\Gamma$ are adjacent in $\Cay(\Gamma, S)$ if and only if $x^{-1}y \in S$. Here, we note that we do not require $S$ to generate $\Gamma$. The following theorem by Payan is an early result illustrating the restrictions that the group structure can impose on the $3$-colorability of a Cayley graph.

\begin{theorem}[{Payan \cite{Payan}; see also \cite{BNT,CK}}]\label{theorem Payan}
Let $m$ be a non-negative integer. If $\Gamma = (\ZZ / 2 \ZZ)^m$ and $S$ is a symmetric subset of $\Gamma$, then $\Cay(\Gamma, S)$ is not $3$-chromatic .
\end{theorem}

Payan's theorem suggests the problem of characterizing algebraically the groups that admit no $3$-chromatic Cayley graph. For abelian groups, Krebs and Sankar \cite{KS} obtained the following algebraic characterization in terms of the exponent of groups. Recall that the \emph{exponent of a group $\Gamma$} is the least positive integer $n$ such that $g^n = 1$ for every $g \in \Gamma$. However, if no such $n$ exists, we set $\exp(\Gamma) = \infty$.

\begin{theorem}[{Krebs--Sankar \cite[Theorems~1.2 and 5.1]{KS}}]\label{theorem KS}
For an abelian group $\Gamma$, the following are equivalent:
\begin{enumerate}[(1)]
\item $\exp(\Gamma) \in \{1,2,4\}$.
\item There is no symmetric subset $S$ of $\Gamma$ such that $\Cay(\Gamma, S)$ is $3$-chromatic.
\end{enumerate}
Moreover, if $\Gamma$ is finitely generated, then the following is also equivalent to the above:
\begin{enumerate}[(3)]
\item There is no symmetric generating set $S$ of $\Gamma$ such that $\Cay(\Gamma, S)$ is $3$-chromatic.
\end{enumerate}
\end{theorem}

We refine Theorem~\ref{theorem KS} using the circular chromatic number. The circular chromatic number $\chi_c(G)$, introduced by Vince \cite{Vince}, is a refinement of the ordinary chromatic number through the relation $\lceil \chi_c(G) \rceil = \chi(G)$ (see Section~\ref{section preliminaries} for the definition). For comprehensive introductions to circular coloring, see \cite{Nadolski,ZhuSurvey,ZhuRecent}. We are now ready to state our main result:

\begin{theorem}\label{theorem main}
Let $k \ge 2$ be an integer. For an abelian group $\Gamma$, the following are equivalent:
\begin{enumerate}[(1)]
\item $\exp(\Gamma) \in \{ 1, \cdots, k\} \cup \{ 2,4, \cdots, 2k\}$.
\item There is no symmetric subset $S$ of $\Gamma$ such that
\[ 2 < \chi_c(\Cay(\Gamma, S)) < 2 + \frac{2}{k-1}.\]
\item There is no symmetric subset $S$ of $\Gamma$ such that
\[ 2 < \chi_c(\Cay(\Gamma, S)) \le 2 + \frac{2}{k}.\]
\end{enumerate}
Moreover, if $\Gamma$ is finitely generated, then the following is also equivalent to the above:
\begin{enumerate}[(4)]
\item There is no symmetric generating set $S$ of $\Gamma$ such that
\[ 2 < \chi_c(\Cay(\Gamma, S)) \le 2 + \frac{2}{k}.\]
\end{enumerate}
\end{theorem}

Note that García-Marco, Knauer, and Menara \cite{GKM} recently proved that the circular chromatic number of every semiminimal Cayley graph of an abelian or generalized dihedral group is at most $4$. We also note that the implication from (1) to (2) in the case $k=2$ of Theorem~\ref{theorem main} was already proved in our previous paper (see \cite[Theorem~6.2]{EM}).

Theorem~\ref{theorem main} follows from Theorems~\ref{theorem existence}, \ref{theorem fg} and~\ref{theorem other direction} below. Our proof uses the \emph{$k$-fundamental group} $\pi_1^k(G,v)$ of a based graph $(G,v)$, introduced by the second author \cite{MatsushitaJMSUT}. For a positive integer $k$, the $k$-fundamental group $\pi_1^k(G,v)$ is a group associated to a based graph $(G,v)$. The $k$-fundamental group is related to Lov\'asz's neighborhood complex \cite{MatsushitaJMSUT}, and is known to have notable applications to topological graph theory \cite{EM}, multiplicative graphs \cite{WrochnaJCTB}, and Hom complexes of graphs \cite{MatsushitaJCTB}. In particular, the $2$-fundamental group gives an obstruction to the $3$-colorability, and Krebs and Sankar \cite{KS} proved Theorem~\ref{theorem KS} obtaining a presentation of the $2$-fundamental group of a Cayley graph of an abelian group. We extend their presentation to $k$-fundamental groups and combine it with the circular-coloring obstruction established in \cite{EM}.

As was mentioned, the $2$-fundamental group $\pi_1^2(G,v)$ is closely related to the fundamental group of the neighborhood complex $\NNN(G)$ introduced by Lov\'asz \cite{Lovasz} in his celebrated proof of Kneser's conjecture. As an application of the $2$-fundamental groups, Krebs and Sankar \cite{KS} proved the following result.

\begin{theorem}[{Krebs--Sankar \cite{KS}}]\label{theorem neighborhood}
Let $G$ be a connected non-bipartite graph. Suppose that $H_1(\NNN(G); \ZZ)$ is a torsion group. Then $G$ is not $3$-colorable.
\end{theorem}

In Appendix~\ref{appendix neighborhood}, we give an alternative short proof of Theorem~\ref{theorem neighborhood} using box complexes \cite{DM} and equivariant maps.

\medskip
\noindent\textbf{Organization of the paper.}
Section~\ref{section preliminaries} reviews our conventions for graphs, circular chromatic numbers, and $k$-fundamental groups, including the circular-coloring obstruction used in the proof. In Section~\ref{section existence}, we construct Cayley graphs with small circular chromatic number and prove Theorem~\ref{theorem existence}. Section~\ref{section fundamental groups} gives a presentation of the $k$-fundamental group of a Cayley graph of an abelian group and applies it to prove Theorem~\ref{theorem other direction}, completing the proof of Theorem~\ref{theorem main}. Appendix~\ref{appendix neighborhood} contains the alternative topological proof of Theorem~\ref{theorem neighborhood}.

\section{Preliminaries}\label{section preliminaries}

We recall the definitions and facts needed in the sequel. For general background on graphs and graph homomorphisms, we refer the reader to \cite{GR,HN}.

\subsection{Graphs}

A graph is a pair $G = (V(G), E(G))$ consisting of a set $V(G)$ and a symmetric subset $E(G)$ of $V(G) \times V(G)$; thus, $(v,w) \in E(G)$ implies $(w,v) \in E(G)$. Our graphs are undirected and have no multiple edges, but they may have loops and need not have finitely many vertices.

For graphs $G$ and $H$, a \emph{graph homomorphism} from $G$ to $H$ is a map $f \colon V(G) \to V(H)$ such that $(v,w) \in E(G)$ implies $(f(v), f(w)) \in E(H)$. Let $K_n$ denote the complete graph on $n$ vertices: $V(K_n) = \{ 1, \cdots, n\}$, and vertices $i,j \in V(K_n)$ are adjacent if and only if $i \ne j$. An \emph{$n$-coloring} of a graph $G$ is a graph homomorphism from $G$ to $K_n$. For $n \ge 3$, the \emph{cycle graph} $C_n$ is defined by $V(C_n) = \ZZ / n\ZZ$ and $E(C_n) = \{ (x,x \pm 1) \mid x \in \ZZ / n\ZZ\}$.

\subsection{Circular chromatic numbers}

Let $S^1 = \RR / \ZZ$, equipped with the distance induced by the Euclidean distance $d_\RR$ on $\RR$. Thus, for $\alpha , \beta \in S^1$, we write
\[ d(\alpha, \beta) = \inf \{ d_\RR(x,y) \mid x \in \alpha, y \in \beta \}.\]
For $0 < a$, the \emph{Borsuk graph} $\Bor(S^1, a)$ is defined by
\[ V(\Bor(S^1, a)) = S^1, \quad E(\Bor(S^1, a)) = \{ (\alpha, \beta) \in S^1 \times S^1 \mid d(\alpha, \beta) \ge a\}.\]
If $a \le b$, then $\Bor(S^1, b)$ is a subgraph of $\Bor(S^1, a)$. Consequently, for $0 < r \le s$, the inequality $s^{-1} \le r^{-1}$ implies that $\Bor(S^1, r^{-1})$ is a subgraph of $\Bor(S^1, s^{-1})$.

\begin{definition}
The \emph{circular chromatic number} of a graph $G$ is
\[ \chi_c(G) = \inf \{ r \mid \textrm{There is a graph homomorphism from $G$ to $\Bor(S^1, r^{-1})$}\}.\]
\end{definition}

For example, odd cycles satisfy
\[ \chi_c(C_{2k+1}) = 2 + \frac{1}{k}.\]
The circular and ordinary chromatic numbers are related by $\chi(G) = \lceil \chi_c(G) \rceil$. Thus, the circular chromatic number is a refinement of the ordinary chromatic number.

We will also use the formulation of circular chromatic number in terms of circular complete graphs. Let $m$ and $n$ be positive integers with $m \ge 2n$. The \emph{circular complete graph} $K_{m/n}$ has vertex set $V(K_{m / n}) = \ZZ / m\ZZ$ and edge set
\[ E(K_{m/n}) = \{ (x,y) \mid d(x,y) \ge n \},\]
where $d(x,y) = \inf \{ d_\RR (a,b) \mid a \in x, b \in y \}$. In terms of these graphs, the circular chromatic number is given by
\[ \chi_c(G) = \inf \left\{ \frac{m}{n} \mid \text{There is a graph homomorphism from $G$ to $K_{m/n}$} \right\}.\]

A basic property of circular complete graphs is $\chi_c(K_{m/n}) = m / n$. The graphs $K_{2k/(k-1)}$ play a central role in this paper. We will use the following descriptions of them.

\begin{lemma}\label{lemma ladder}
Let $k$ be an integer with $k \ge 2$. Then the following hold:
\begin{enumerate}[(1)]
\item If $k$ is odd, then $K_{2k /(k-1)}$ is the graph $C_k \square K_2$, where $\square$ denotes the Cartesian product.
\item If $k$ is even, then $K_{2k/(k-1)}$ is the M\"obius ladder $M_{2k}$ on $2k$ vertices.
\end{enumerate}
\end{lemma}

Here, the \emph{M\"obius ladder} $M_{2k}$ is the graph with vertex set $\ZZ / 2k \ZZ$ and edge set
\[ E(M_{2k}) = \{ (x,y) \mid x - y \in \{\pm 1, k\}\}.\]
Lemma~\ref{lemma ladder} gives $\chi_c(M_{2k}) = 2k / (k-1)$ when $k$ is even. When $k$ is odd, there are graph homomorphisms in both directions between $K_{2k / (k-1)} \cong C_k \square K_2$ and $C_k$. Hence we have $\chi_c(C_k) = 2k / (k-1)$. We mention these formulas for later use.

\begin{lemma}\label{lemma cycle ladder}
Let $m$ be an integer such that $m \ge 2$. Then the following hold:
\begin{enumerate}[(1)]
\item If $m$ is odd, then $\chi_c(C_m) = 2 + \frac{2}{m-1}$.
\item If $m$ is even, then $\chi_c(M_{2m}) = 2 + \frac{2}{m - 1}$.
\end{enumerate}
\end{lemma}

\subsection{\texorpdfstring{$k$}{k}-fundamental groups}\label{subsection k-fundamental group}

We recall the definition of the $k$-fundamental group $\pi_1^k(G,v)$, introduced in \cite{MatsushitaJMSUT}, and the circular-coloring obstruction that will be used below.

Let $P_n$ be the \emph{path graph} on $n + 1$ vertices, with $V(P_n) = \{ 0,1, \cdots, n\}$ and $E(P_n) = \{ (i,j) \mid |i-j| = 1\}$.

A \emph{based graph} is a pair $(G,v)$ consisting of a graph $G$ and a distinguished vertex $v$. A \emph{closed walk} of length $n$ in $(G,v)$ is a graph homomorphism $\gamma \colon P_n \to G$ such that $\gamma(0) = \gamma(n) = v$. We denote its length by $\length(\gamma)$ and write $\Omega(G,v)$ for the set of closed walks in $(G,v)$. For a positive integer $k$, let $\simeq_k$ be the equivalence relation generated by the following two types of elementary moves between $\gamma, \gamma' \in \Omega(G,v)$:
\begin{itemize}
\item[(A)] $\length(\gamma') = \length(\gamma) + 2$, and there is $i \in \{ 0,1, \cdots, \length(\gamma) \}$ such that $\gamma(j) = \gamma'(j)$ for $j \le i$ and $\gamma(j) = \gamma'(j + 2)$ for $j \ge i$.
\item[(B)$_k$] $\length(\gamma) = \length(\gamma')$ and $\# \{ i \in \{ 0,1, \cdots, \length (\gamma) \} \mid \gamma(i) \ne \gamma'(i) \} < k$.
\end{itemize}

More explicitly, $\gamma \simeq_k \gamma'$ if and only if there is a sequence $(\gamma_0, \cdots, \gamma_m)$ with $\gamma = \gamma_0$ and $\gamma' = \gamma_m$ such that, for every $i \in \{ 1, \cdots, m\}$, either $(\gamma_{i-1}, \gamma_i)$ or $(\gamma_i, \gamma_{i-1})$ satisfies (A) or (B)$_k$. The equivalence class of $\gamma$ is denoted by $[\gamma]_k$ and is called its \emph{$k$-homotopy class}.

\begin{definition}
Let $k$ be a positive integer. The \emph{$k$-fundamental group} of a based graph $(G,v)$ is the quotient set $\pi_1^k(G,v) = \Omega(G,v) / \simeq_k$, equipped with the operation induced by concatenation of closed walks.
\end{definition}

\begin{remark} \label{remark consecutive}
For $\gamma, \gamma' \in \Omega(G,v)$, we consider the following condition:
\begin{itemize}
\item[(B)$'_k$] $\length(\gamma) = \length(\gamma')$ and there is $i \in \{ 0,1, \cdots, \length(\gamma)\}$ and $m < k$ such that $\gamma(j) = \gamma'(j)$ for $j \not\in \{ i+1, \cdots, i+m \}$.
\end{itemize}
Then it is straightforward to see that $\simeq_k$ is generated by the relations (A) and (B)$'_k$.
\end{remark}

To describe this operation, let $\gamma \colon P_m \to G$ and $\gamma' \colon P_n \to G$ be closed walks in $(G,v)$. Their \emph{concatenation} is the closed walk $\gamma \cdot \gamma' \colon P_{m+n} \to G$ of length $m + n$ defined by
\[ \gamma \cdot \gamma'(i) = \begin{cases}
\gamma(i) & (i \le m) \\
\gamma'(i-m) & (i \ge m).
\end{cases}\]
Concatenation induces the group operation on $\pi_1^k(G,v)$.

The following result relates torsion in the abelianization of the $k$-fundamental group to circular chromatic number.

\begin{theorem}[{\cite[Theorem~5.3]{EM}}]\label{theorem circular}
Let $G$ be a connected graph and let $k \ge 2$ be an integer. Suppose that there is a closed walk $\gamma$ of odd length in $G$ such that the image of $[\gamma]_k \in \pi_1^k(G)$ in the abelianization $\pi_1^k(G) / [\pi_1^k(G), \pi_1^k(G)]$ is a torsion element. Then
\[ \chi_c(G) \ge 2 + \frac{2}{k-1}.\]
\end{theorem}

\section{Cayley graphs with small circular chromatic number}\label{section existence}

In this section, we establish the existence results (Theorems~\ref{theorem existence} and \ref{theorem fg}) used in the proof of Theorem~\ref{theorem main}. Recently, García-Marco, Knauer, and Menara \cite{GKM} proved that the circular chromatic number of every semiminimal Cayley graph of an abelian group is at most $4$, and Lemma~\ref{lemma odd} in this paper is deduced from their results (see \cite[Theorem~4.3]{GKM}). Nevertheless, for the convenience of the reader, we provide direct proofs here.

\begin{theorem}\label{theorem existence}
Let $k \ge 2$ be an integer, and let $\Gamma$ be an abelian group such that
\[ \exp(\Gamma) \not\in \{ 1,2, \cdots, k\} \cup \{ 2,4, \cdots, 2k\}.\]
Then there is a symmetric subset $S$ of $\Gamma$ such that
\[ 2 < \chi_c(\Cay(\Gamma, S)) \le 2 + \frac{2}{k}.\]
\end{theorem}

In Theorem~\ref{theorem fg}, we show that if $\Gamma$ is a finitely generated abelian group satisfying the assumption of Theorem~\ref{theorem existence} then $S$ can be chosen to be a generating set of $\Gamma$.

We note that when $\Gamma$ is an abelian group, then $\exp(\Gamma)$ coincides with the supremum of the order of elements in $\Gamma$. Hence Theorem~\ref{theorem existence} is deduced from Lemmas~\ref{lemma infinite}, \ref{lemma even}, \ref{lemma odd} below.

\begin{lemma}\label{lemma infinite}
Let $\Gamma$ be a group containing an element $x$ of infinite order. For every $\varepsilon > 0$, there is a symmetric subset $S$ of $\Gamma$ such that $2 < \chi_c(\Cay(\Gamma, S)) < 2 + \varepsilon$.
\end{lemma}
\begin{proof}
Choose an even integer $k$ such that $2/(k-1) < \varepsilon$, and set $S = \{ x^{\pm 1}, x^{\pm k}\}$. Then $\Cay(\Gamma, S)$ is a disjoint union of copies of the graph $G_k$ defined by
\[ V(G_k) = \ZZ, \quad E(G_k) = \{ (x,y) \mid |x-y| = 1, k\}.\]
Since the graph $G_k$ contains the cycle $C_{k+1}$ as a subgraph, $\Cay(\Gamma, S)$ is non-bipartite. Moreover, the quotient map $p \colon \ZZ \to \ZZ / 2k$ induces a graph homomorphism $p \colon G_k \to M_{2k}$. Thus $\chi_c(\Cay(\Gamma, S)) \le \chi_c(M_{2k})$, and, since $k$ is even, we obtain
\[ \chi_c(\Cay(\Gamma, S)) \le \chi_c(M_{2k}) = \frac{2k}{k-1} = 2 + \frac{2}{k-1} < 2 + \varepsilon.\qedhere\]
\end{proof}

\begin{lemma}\label{lemma even}
Let $\Gamma$ be a group containing an element $x$ of even order $2m$ with $m > 1$. Then there is a symmetric subset $S$ of $\Gamma$ such that
\[ \chi_c(\Cay(\Gamma, S)) = 2 + \frac{2}{m-1}.\]
\end{lemma}
\begin{proof}
First suppose that $m = 2n$. Set $S = \{ x^{\pm 1}, x^{\pm 2n}\}$. Then $\Cay(\Gamma, S)$ is a disjoint union of the M\"obius ladder $M_{2m}$, and Lemma~\ref{lemma cycle ladder} gives
\[ \chi_c(\Cay(\Gamma, S)) = \chi_c(M_{2m}) = 2 + \frac{2}{m-1}.\]
Next suppose that $m$ is odd. Taking $S = \{ x^{\pm 2}\}$, we see that $\Cay(\Gamma, S)$ is a disjoint union of copies of $C_m$. Hence, by Lemma~\ref{lemma cycle ladder},
\[ \chi_c(\Cay(\Gamma, S)) = \chi_c(C_m) = 2 + \frac{2}{m-1}.\qedhere\]
\end{proof}

\begin{lemma}\label{lemma odd}
Let $\Gamma$ be a group containing an element $x$ of odd order $m$ with $m > 1$. Then there is a symmetric subset $S$ of $\Gamma$ such that
\[ \chi_c(\Cay(\Gamma, S)) = 2 + \frac{2}{m-1}.\]
\end{lemma}
\begin{proof}
Set $S = \{ x^{\pm 1}\}$. Then $\Cay(\Gamma, S)$ is a disjoint union of copies of $C_m$, and the conclusion follows from Lemma~\ref{lemma cycle ladder}.
\end{proof}

\begin{proof}[Proof of Theorem~\ref{theorem existence}]
Choose $x \in \Gamma$ whose order does not belong to $\{ 1, \cdots, k\} \cup \{ 2,4, \cdots, 2k\}$. If $x$ has infinite order, apply Lemma~\ref{lemma infinite}. If its order is even and greater than $2k$, apply Lemma~\ref{lemma even}. If its order is odd and greater than $k$, apply Lemma~\ref{lemma odd}. These cases establish the result.
\end{proof}

Krebs and Sankar showed that if $\Gamma$ is a finitely generated abelian group with $\exp(\Gamma) \notin \{1,2,4\}$, then there exists a generating set $S$ of $\Gamma$ such that $\Cay(\Gamma,S)$ is $3$-chromatic (see \cite[Theorem~5.1]{KS}). Similarly, if $\Gamma$ is a finitely generated abelian group, then the set $S$ in Theorem~\ref{theorem existence} can be chosen to be a generating set of $\Gamma$.

\begin{theorem} \label{theorem fg}
Let $k \ge 2$ be an integer, and let $\Gamma$ be a finitely
generated abelian group such that
\[
\exp(\Gamma) \notin
\{1,2,\ldots,k\} \cup \{2,4,\ldots,2k\}.
\]
Then there exists a symmetric generating set $S$ of $\Gamma$
such that
\[
2 < \chi_c(\Cay(\Gamma,S)) \le 2+\frac{2}{k}.
\]
\end{theorem}

\begin{proof}
First, suppose that $\Gamma$ is infinite.
We may assume that
\[
\Gamma = \Gamma_1 \times \cdots \times \Gamma_n,
\]
where each $\Gamma_i$ is cyclic and $\Gamma_1 \cong \ZZ$.
Let $x_i$ be a generator of $\Gamma_i$ for each $i$,
and choose an even integer $m$ such that $m-1 \ge k$.
Define a symmetric generating set $S$ of $\Gamma$ by
\[
S =
\{x_1^{\pm1},x_1^{\pm m}\}
\cup
\{(x_1x_i)^{\pm1} \mid i=2,\ldots,n\}.
\]
Since $\Cay(\Gamma,S)$ contains an odd cycle of length
$m+1$, it is non-bipartite.
Moreover, the composition of the projection
$\Gamma \to \Gamma_1 \cong \ZZ$
and the quotient map
$\ZZ \to \ZZ/2m\ZZ$
induces a graph homomorphism
from $\Cay(\Gamma,S)$ to $M_{2m}$.
Consequently,
\[
\chi_c(\Cay(\Gamma,S))
\le \chi_c(M_{2m})
= \frac{2m}{m-1}
= 2+\frac{2}{m-1}
\le 2+\frac{2}{k}.
\]

Next, suppose that $\Gamma$ is finite.
Let $p_1<\cdots<p_r$ be the distinct prime divisors
of $|\Gamma|$, and let $q_i$ be the maximum order
of an element of the $p_i$-primary component of $\Gamma$.
Set
\[
\Gamma'_i=\ZZ/q_i\ZZ.
\]
Then we can write
\[
\Gamma \cong
\Gamma'_1\times\cdots\times\Gamma'_r
\times\Gamma_2\times\cdots\times\Gamma_n,
\]
where $\Gamma_2,\ldots,\Gamma_n$ are cyclic groups.
Put $q=\exp(\Gamma)$. Then $q=q_1\cdots q_r$.
Since the orders $q_1,\ldots,q_r$ are pairwise coprime,
the group
\[
\Gamma_1=\Gamma'_1\times\cdots\times\Gamma'_r
\]
is cyclic and isomorphic to $\ZZ/q\ZZ$.
Thus, we identify $\Gamma$ with
$\Gamma_1\times\cdots\times\Gamma_n$.
For each $i$, let $x_i$ be a generator of $\Gamma_i$.

Suppose first that $q=\exp(\Gamma)$ is odd.
Then $q>k$.
Define a symmetric generating set $S$ of $\Gamma$ by
\[
S=
\{x_1^{\pm1}\}
\cup
\{(x_1x_i)^{\pm1}\mid i=2,\ldots,n\}.
\]
Since $\Cay(\Gamma,S)$ contains an odd cycle of length $q$,
it is non-bipartite.
Moreover, the projection
$\Gamma\to\Gamma_1\cong\ZZ/q\ZZ$
induces a graph homomorphism
from $\Cay(\Gamma,S)$ to $C_q$.
Therefore,
\[
\chi_c(\Cay(\Gamma,S))
\le \chi_c(C_q)
=2+\frac{2}{q-1}
\le 2+\frac{2}{k}.
\]

Finally, suppose that $q=\exp(\Gamma)$ is even.
Write $q=2q'$. Then $q'>k$.

If $q'$ is even, define a symmetric generating set
$S$ of $\Gamma$ by
\[
S=
\{x_1^{\pm1},x_1^{\pm q'}\}
\cup
\{(x_1x_i)^{\pm1}\mid i=2,\ldots,n\}.
\]
Since $\Cay(\Gamma,S)$ contains an odd cycle of length
$q'+1$, it is non-bipartite.
Moreover, the projection
$\Gamma\to\Gamma_1\cong\ZZ/2q'\ZZ$
induces a graph homomorphism
from $\Cay(\Gamma,S)$ to $M_{2q'}$.
Hence,
\[
\chi_c(\Cay(\Gamma,S))
\le \chi_c(M_{2q'})
=2+\frac{2}{q'-1}
\le 2+\frac{2}{k}.
\]

If $q'$ is odd, define a symmetric generating set
$S$ of $\Gamma$ by
\[
S=
\{x_1^{\pm2},x_1^{\pm q'}\}
\cup
\{(x_1^2 x_i)^{\pm1}\mid i=2,\ldots,n\}.
\]
Since $\Cay(\Gamma,S)$ contains an odd cycle of length
$q'$, it is non-bipartite.
Moreover, the projection $\Gamma\to\Gamma_1$
induces a graph homomorphism
\[
\Cay(\Gamma,S)
\longrightarrow
\Cay(\Gamma_1,\{x_1^{\pm2},x_1^{\pm q'}\})
\cong C_{q'}\square K_2.
\]
Since $\chi_c(C_{q'}\square K_2)=\chi_c(C_{q'})$, we obtain
\[
\chi_c(\Cay(\Gamma,S))
\le \chi_c(C_{q'})
=2+\frac{2}{q'-1}
\le 2+\frac{2}{k}.
\]
This completes the proof.
\end{proof}

\section{Lower bound for the circular chromatic numbers}\label{section fundamental groups}

We now prove the lower bound needed to complete the proof of Theorem~\ref{theorem main}. We first give a presentation of the $k$-fundamental group of a Cayley graph of an abelian group, extending the presentation of Krebs and Sankar \cite{KS}.

Let $\Gamma$ be an abelian group, written multiplicatively with identity $1_\Gamma$, and let $S$ be a symmetric subset of $\Gamma$. Let $\ZZ[S]$ be the free abelian group generated by $S$, and write $e_s$ for the generator corresponding to $s \in S$. Define $\varphi \colon \ZZ[S] \to \Gamma$ by
\[ \varphi(a_1 e_{s_1} + \cdots + a_m e_{s_m}) = s_1^{a_1} \cdots s_m^{a_m}.\]
Set $R[S] = \Ker(\varphi)$. For $k \ge 1$, let $H_k[S]$ be the subgroup of $R[S]$ generated by the following elements:
\begin{enumerate}[(i)]
\item $e_s + e_{s^{-1}}$ for every $s \in S$;
\item $(e_{s_1} + \cdots + e_{s_m}) - (e_{s'_1} + \cdots + e_{s'_m})$ for all $s_1, \cdots, s_m, s'_1, \cdots, s'_m$ and for $m \le k$ such that $s_1 \cdots s_m = s'_1 \cdots s'_m$.
\end{enumerate}

Krebs and Sankar \cite{KS} proved that $\pi_1^2(\Cay(\Gamma, S), 1_\Gamma) \cong R[S]/H_2[S]$. We generalize their result as follows.

\begin{theorem}\label{theorem presentation}
Let $k \ge 2$ be an integer, let $\Gamma$ be an abelian group, and let $S$ be a symmetric subset of $\Gamma$. Then $\pi_1^k(\Cay(\Gamma, S), 1_\Gamma) \cong R[S] / H_k[S]$. In particular, $\pi_1^k(\Cay(\Gamma, S), 1_\Gamma)$ is abelian.
\end{theorem}
\begin{proof}
Let $M[S]$ be the set of finite sequences of elements of $S$, with the monoid operation given by concatenation. We describe closed walks in $\Cay(\Gamma,S)$ based at $1_\Gamma$ by their successive steps. Let $s_1, \cdots, s_\ell \in S$ satisfy $s_1 \cdots s_\ell = 1_\Gamma$ in $\Gamma$, and, for $i = 0,1, \cdots, \ell$, set
\[ v_i = s_1 \cdots s_i. \]
Then $(v_0, \cdots, v_\ell)$ is a closed walk in $\Cay(\Gamma,S)$ based at $1_\Gamma$. We denote this walk by $(s_1, \cdots, s_\ell)$. In this way, we identify $\Omega(\Cay(\Gamma, S), 1_\Gamma)$ with the subset
\[ \hatR[S] = \{ (s_1, \cdots, s_\ell) \in M[S] \mid s_1 \cdots s_\ell = 1_\Gamma \}\]
of $M[S]$. We write $[s_1, \cdots, s_\ell]$ for the $k$-homotopy class of the corresponding closed walk $(s_1, \cdots, s_\ell) \in \hatR[S]$. The definition of $k$-homotopy gives the following moves:
\begin{enumerate}[(a)]
\item If $s_{i+1} = s_i^{-1}$, deleting the indicated pair of steps does not change the class:
\[ [s_1, \cdots, s_\ell]_k = [s_1, \cdots, s_{i-1}, s_{i+2}, \cdots, s_{\ell}]_k. \]
\item If $m \le k$ and $s_{i+1} \cdots s_{i+m} = s'_{i+1} \cdots s'_{i+m}$, then
\[ [s_1, \cdots, s_\ell]_k = [s_1, \cdots,s_i, s'_{i+1}, \cdots, s'_{i+m}, s_{i+m+1}, \cdots, s_\ell]_k.\]
\end{enumerate}
Since $k \ge 2$ and $\Gamma$ is abelian, move (b) shows that interchanging two consecutive steps does not change the $k$-homotopy class:
\begin{equation}\label{equation transposition}
[s_1, \cdots, s_m]_k = [s_1, \cdots, s_{i-1}, s_{i+1}, s_i, s_{i+2}, \cdots, s_m]_k.
\end{equation}
Define $\hPsi \colon \hatR[S] \to \pi_1^k(\Cay(\Gamma, S), 1_\Gamma)$ to be the composite
\[ \hatR[S] \xrightarrow{\cong} \Omega(\Cay(\Gamma, S), 1_\Gamma) \to \pi_1^k(\Cay(\Gamma, S), 1_\Gamma).\]
By the definition of the product, $\hPsi$ is a monoid homomorphism.

Let $q \colon \hatR[S] \to R[S]$ be the natural map given by
\[ q(s_1, \cdots, s_m) = e_{s_1} + \cdots + e_{s_m}. \]
We construct a group homomorphism $\Psi \colon R[S] \to \pi_1^k(\Cay(\Gamma, S), 1_\Gamma)$ such that $\Psi \circ q = \hPsi$.

For $x = a_1 e_{s_1} + \cdots + a_m e_{s_m} \in R[S]$, with each of $a_1, \cdots, a_m$ nonzero, define $\Psi \colon R[S] \to \pi_1^k(\Cay(\Gamma, S), 1_\Gamma)$ by
\[
\begin{split}
\Psi(x) = \big[\,&\underbrace{s_1^{\sgn(a_1)}, \cdots, s_1^{\sgn(a_1)}}_{|a_1|\text{ times}},
\underbrace{s_2^{\sgn(a_2)}, \cdots, s_2^{\sgn(a_2)}}_{|a_2|\text{ times}},\\
&\cdots, \underbrace{s_m^{\sgn(a_m)}, \cdots, s_m^{\sgn(a_m)}}_{|a_m|\text{ times}} \big].
\end{split}
\]
The cancellation move and \eqref{equation transposition} show that $\Psi(x)$ is independent of the expression $x = a_1 e_{s_1} + \cdots + a_m e_{s_m}$. This independence also shows that $\Psi$ is a group homomorphism. The identity $\Psi \circ q = \hPsi$ follows directly from the definition. Since $\hPsi$ is a composite of surjections, it is surjective, and hence so is $\Psi$.

It remains to prove that $\Ker(\Psi) = H_k[S]$. We first show that $H_k[S] \subset \Ker(\Psi)$. For each $s \in S$, we have
\[ \Psi(e_s + e_{s^{-1}}) = [s, s^{-1}]_k = [\ ]_k.\]
Thus every generator of type (i) in the definition of $H_k[S]$ lies in $\Ker(\Psi)$. Next, let $s_1, \cdots, s_m, t_1, \cdots, t_m \in S$ with $m \le k$ satisfy $s_1 \cdots s_m = t_1 \cdots t_m$. Then
\begin{align*}
\Psi\big( (e_{s_1} + \cdots + e_{s_m}) - (e_{t_1} + \cdots + e_{t_m})\big)
&= [s_1, \cdots, s_m, t_1^{-1}, \cdots, t_m^{-1}]_k \\
&= [t_m, \cdots, t_1, t_1^{-1}, \cdots, t_m^{-1}]_k\\
&= [\ ]_k.
\end{align*}
Hence every generator of type (ii) also lies in $\Ker(\Psi)$, proving $H_k[S] \subset \Ker(\Psi)$.

For the reverse inclusion, take $x \in \Ker(\Psi)$ and choose $s_1, \cdots, s_\ell, t_1, \cdots, t_m \in S$ such that
\[ x = (e_{s_1} + \cdots + e_{s_\ell}) - (e_{t_1} + \cdots + e_{t_m}). \]
Then we have
\[ x = (e_{s_1} + \cdots + e_{s_\ell} + e_{t_1^{-1}} + \cdots + e_{t_m^{-1}}) - (e_{t_1} + e_{t_1^{-1}} + \cdots + e_{t_m} + e_{t_m^{-1}}).\]
Since $e_{t_i} + e_{t_i^{-1}} \in H_k[S]$, we can assume that $x$ is written as
\[ x = e_{s_1} + \cdots + e_{s_\ell},\]
and show $x \in H_k[S]$. Set $\hx = (s_1, \cdots, s_\ell) \in \hatR[S] \cong \Omega(\Cay(\Gamma, S), 1_\Gamma)$, so that $q(\hx) = x$. Since $x \in \Ker(\Psi)$, there is a sequence in $\hatR[S] \cong \Omega(\Cay(\Gamma, S), 1_\Gamma)$,
\[ \hx = \hx_0, \cdots, \hx_n = (),\]
such that, for every $i \in \{ 1, \cdots, n\}$, either $(\hx_{i-1}, \hx_i)$ or $(\hx_i, \hx_{i-1})$ satisfies (A) or (B)$'_k$ (see Remark~\ref{remark consecutive}). Since $q(\hx_n)=0_{R[S]}$, it suffices to show that $q(\hx_{i-1}) - q(\hx_i) \in H_k[S]$ for every $i \in \{ 1, \cdots, n\}$; this will give $x = q(\hx) = q(\hx_0) \in H_k[S]$.

If the move is of type (A), then there is an $s \in S$ such that $q(\hx_{i-1}) - q(\hx_i) = \pm(e_s + e_{s^{-1}})$, which belongs to $H_k[S]$. If the move is of type (B)$'_k$, there are $m \le k$ and $t_1, \cdots, t_m, t'_1, \cdots, t'_m \in S$ with $t_1 \cdots t_m = t'_1 \cdots t'_m$ in $\Gamma$ such that
\[ q(\hx_i) - q(\hx_{i-1}) = (e_{t'_1} + \cdots + e_{t'_m}) - (e_{t_1} + \cdots + e_{t_m}). \]
The right-hand side belongs to $H_k[S]$. This completes the proof of $\Ker(\Psi) = H_k[S]$.
\end{proof}

We now apply the presentation to obtain the other direction of Theorem~\ref{theorem main}.

\begin{theorem}\label{theorem other direction}
Let $k \ge 2$ be an integer, and let $\Gamma$ be an abelian group such that
\[ \exp(\Gamma) \in \{ 1, \cdots, k\} \cup \{ 2,4, \cdots, 2k\}.\]
Then, for every symmetric subset $S$ of $\Gamma$ such that $\Cay(\Gamma, S)$ is non-bipartite, we have
\[ \chi_c(\Cay(\Gamma, S)) \ge 2 + \frac{2}{k-1}.\]
\end{theorem}
\begin{proof}
Let $G_0$ be the connected component of $\Cay(\Gamma, S)$ containing $1_\Gamma$. Since $\Cay(\Gamma, S)$ is a disjoint union of copies of $G_0$, the graph $G_0$ is non-bipartite. As $G_0$ is connected and non-bipartite, it contains a closed walk
\[ \gamma = (s_1, \cdots, s_m) \]
based at $1_\Gamma$ whose length $m$ is odd. Define $n = \exp(\Gamma) / 2$ if $\exp(\Gamma) \in \{ 2,4, \cdots, 2k\}$, and define $n = \exp(\Gamma)$ if $\exp(\Gamma) \not\in \{ 2,4, \cdots, 2k\}$. Let $\Phi \colon \pi_1^k(\Cay(\Gamma, S), 1_\Gamma) \to R[S] / H_k[S]$ be the isomorphism from Theorem~\ref{theorem presentation}. Then
\[ \Phi([\gamma]_k^{2n}) = 2n e_{s_1} + \cdots + 2n e_{s_m}. \]
For each $i \in \{ 1, \cdots, m\}$, let $\gamma_i$ be the closed walk in $\Cay(\Gamma, S)$ given by
\[ \gamma_i = (\underbrace{s_i, \cdots, s_i}_{2n\text{ times}}).\]
We have
\[ \Phi([\gamma_1]_k \cdots [\gamma_m]_k) = 2n e_{s_1} + \cdots + 2n e_{s_m}.\]
Since $\Phi$ is an isomorphism, it follows that $[\gamma_1]_k \cdots [\gamma_m]_k = [\gamma]_k^{2n}$. Moreover, since $n \le k$, we have
\[ \gamma_i = (\underbrace{s_i, \cdots, s_i}_{2n\text{ times}}) \simeq_k (\underbrace{s_i^{-1}, \cdots, s_i^{-1}}_{n{ \text{ times}}}, \underbrace{s_i, \cdots, s_i}_{n{ \text{ times}}}) \simeq_k *_{1_\Gamma}.\]
Thus $[\gamma]_k$ is a torsion element of $\pi_1^k(\Cay(\Gamma, S), 1_\Gamma)$ represented by a closed walk of odd length. The conclusion follows from Theorem~\ref{theorem circular}.
\end{proof}

Now we complete the proof of our main theorem (Theorem~\ref{theorem main}).

\begin{proof}[Proof of Theorem~\ref{theorem main}]
The implication $(3) \Rightarrow (1)$ follows from Theorem~\ref{theorem existence}. The implication $(2) \Rightarrow (3)$ is clear. The implication $(1) \Rightarrow (2)$ follows from Theorem~\ref{theorem other direction}.

Suppose that $\Gamma$ is finitely generated. Then, the implication $(4) \Rightarrow (1)$ follows from Theorem~\ref{theorem fg} and the implication $(3) \Rightarrow (4)$ is clear. This completes the proof.
\end{proof}

\appendix

\section{A topological proof of Theorem~\ref{theorem neighborhood}}\label{appendix neighborhood}

We give an alternative proof of Theorem~\ref{theorem neighborhood} using standard properties of the neighborhood and box complexes. For definitions and basic facts of the neighborhood complex $\NNN(G)$ and the box complex $\BBB(G)$, we refer the reader to \cite{DM, Matousek} (see also \cite{Kozlov, MW}). We recall only the properties needed here:
\begin{itemize}
\item For a simple graph $G$, the box complex $\BBB(G)$ is a free $\ZZ_2$-simplicial complex.
\item The complexes $\NNN(G)$ and $\BBB(G)$ are homotopy equivalent.
\item If $G$ is connected and non-bipartite, then $\NNN(G)$ is connected, as follows directly from its definition.
\end{itemize}
We regard $S^n$ as a $\ZZ_2$-space under the antipodal action. For a free $\ZZ_2$-space $X$, its \emph{$\ZZ_2$-index} is defined by
\[ \ind_{\ZZ_2}(X) = \inf \{ n \ge -1 \mid \text{There is a continuous $\ZZ_2$-map from $X$ to $S^n$}\},\]
where $S^{-1} = \emptyset$. A $\ZZ_2$-map means an equivariant map with respect to the specified actions.

We use the following standard lower bound for the chromatic number.

\begin{theorem}[{see \cite[Theorem~1]{MZ}}]\label{theorem index}
For every graph $G$,
\[ \chi(G) \ge \ind_{\ZZ_2}(\BBB(G)) + 2. \]
\end{theorem}

We also use the classical odd-degree theorem for equivariant self-maps of spheres.

\begin{theorem}[{see \cite[Proposition~2B.6]{Hatcher}}]\label{theorem Borsuk}
Let $f \colon S^n \to S^n$ be a continuous $\ZZ_2$-map. Then $f$ has odd degree. In particular, the induced homomorphism $f_* \colon H_n(S^n; \ZZ) \to H_n(S^n; \ZZ)$ is nontrivial.
\end{theorem}

\begin{proof}[Proof of Theorem~\ref{theorem neighborhood}]
Suppose that $H_1(\NNN(G); \ZZ)$ is a torsion group and, for a contradiction, that $\chi(G) = 3$. Since $G$ is connected and non-bipartite, $\NNN(G)$ is connected. Its homotopy equivalence with $\BBB(G)$ implies that $\BBB(G)$ is connected as well.

There is therefore a $\ZZ_2$-map $f$ from $S^1$ to $\BBB(G)$. Indeed, let $\tau$ denote the involution on $\BBB(G)$, choose a point $a \in \BBB(G)$, and take a path $\gamma$ from $a$ to $\tau(a)$. Joining $\gamma$ to $\tau \circ \gamma$ gives the required equivariant map from $S^1$ to $\BBB(G)$.

By Theorem~\ref{theorem index}, we have $\ind_{\ZZ_2}(\BBB(G)) \le 1$, and hence there is a $\ZZ_2$-map $g \colon \BBB(G) \to S^1$. The composite $g \circ f \colon S^1 \to S^1$ is equivariant. Since $(g \circ f)_* = g_* \circ f_*$, Theorem~\ref{theorem Borsuk} implies that $g_* \colon H_1(\BBB(G); \ZZ) \to H_1(S^1; \ZZ)$ is nontrivial. On the other hand, $H_1(\BBB(G); \ZZ) \cong H_1(\NNN(G); \ZZ)$ is a torsion group, and there is no nontrivial homomorphism from a torsion group to $H_1(S^1; \ZZ) \cong \ZZ$. This is a contradiction.
\end{proof}

\section*{Acknowledgments}
The first and second authors were supported in part by JSPS KAKENHI Grant Numbers JP23K13006 and JP23K12975, respectively. The authors are grateful to Kolja Knauer for bringing his related work to their attention.

\section*{Data availability}
No datasets were generated or analyzed in this study.

\section*{Declaration of generative AI and AI-assisted technologies in the manuscript preparation process}
The authors used ChatGPT and Gemini to improve the language, organization, and readability of the manuscript and to assist in reviewing the relevant literature. We also used ChatGPT 6.0 Pro Astra to check for mathematical consistency. All mathematical results and proofs in this paper were developed by the authors independently of generative AI. The authors reviewed and edited the resulting text and retain full responsibility for the content of the manuscript.


\begin{thebibliography}{99}

\bibitem{Alon}
N.~Alon, The chromatic number of random Cayley graphs, European J. Combin. \textbf{34} (2013), 1232--1243.

\bibitem{AO}
N.~Alon and A.~Orlitsky, Repeated communication and Ramsey graphs, IEEE Trans. Inform. Theory \textbf{41} (1995), 1276--1289.

\bibitem{Babai}
L.~Babai, Chromatic number and subgraphs of Cayley graphs, in \emph{Theory and Applications of Graphs} (Y.~Alavi and D.~R.~Lick, eds.), Lecture Notes in Mathematics, vol.~642, Springer-Verlag, Berlin, 1978, pp.~10--22.

\bibitem{BNT}
L.~Beaudou, R.~Naserasr, and C.~Tardif, Homomorphisms of binary Cayley graphs, Discrete Math. \textbf{338} (2015), 2539--2544.

\bibitem{Czerwinski}
S.~Czerwi\'nski, Random runners are very lonely, J. Combin. Theory Ser. A \textbf{119} (2012), 1194--1199.

\bibitem{CK}
J.~Cervantes and M.~Krebs, Chromatic numbers of Cayley graphs of abelian groups: A matrix method, Linear Algebra Appl. \textbf{676} (2023), 277--295.

\bibitem{DM} H.R.~Daneshpajouh, F.~meunier, Box complexes: At the crossroad of graph theory and topology, Discrete Math. {\bf 348} (2025) 114422.

\bibitem{DHY}
J. Davies, M. Hatzel, and L. Yepremyan, Minimal Cayley graphs with large chromatic number, arXiv:2608.06254

\bibitem{EM}
K.~Enami and T.~Matsushita, Higher-dimensional generalization of Youngs' theorem and circular colorings, J. Combin. Theory Ser. B \textbf{180} (2026), 62--83.

\bibitem{GK}
I.~Garc\'ia-Marco and K.~Knauer, Coloring minimal Cayley graphs, European J. Combin. \textbf{125} (2025), 104108.

\bibitem{GKM}
I. García-Marco, K. Knauer, and G. Menara, Coloring semiminimal Cayley Graphs, arXiv:2607.26942

\bibitem{GR}
C.~Godsil and G.~Royle, \emph{Algebraic Graph Theory}, Graduate Texts in Mathematics, vol.~207, Springer-Verlag, New York, 2001.

\bibitem{GR2}
C. Godsil and B. Rooney, Hardness of computing clique number and chromatic number for Cayley graphs, European Journal of Combinatorics \textbf{62} (2017), 147–166.

\bibitem{Hatcher}
A.~Hatcher, \emph{Algebraic Topology}, Cambridge University Press, Cambridge, 2002.

\bibitem{HN}
P.~Hell and J.~Ne\v{s}et\v{r}il, \emph{Graphs and Homomorphisms}, Oxford Lecture Series in Mathematics and its Applications, vol.~28, Oxford University Press, Oxford, 2004.

\bibitem{Lovasz}
L.~Lov\'asz, Kneser's conjecture, chromatic number, and homotopy, J. Combin. Theory Ser. A \textbf{25} (1978), 319--324.

\bibitem{Kozlov}
D.~N.~Kozlov, \emph{Combinatorial Algebraic Topology}, Algorithms and Computation in Mathematics, vol.~21, Springer, Berlin, 2008.

\bibitem{KS}
M.~Krebs and M.~Sankar, Abelian groups without $3$-chromatic Cayley graphs, J. Combin. Theory Ser. B \textbf{179} (2026), 205--218 \url{https://doi.org/10.1016/j.jctb.2026.03.004}.

\bibitem{Matousek}
J.~Matou\v{s}ek, \emph{Using the Borsuk--Ulam Theorem: Lectures on Topological Methods in Combinatorics and Geometry}, Universitext, Springer-Verlag, Berlin, 2003.

\bibitem{MZ}
J. Matoušek and G. M. Ziegler, Topological lower bounds for the chromatic number: A hierarchy, Jahresbericht der DMV 106 (2004), 71–90.

\bibitem{MatsushitaJMSUT}
T.~Matsushita, Fundamental groups of neighborhood complexes, J. Math. Sci. Univ. Tokyo \textbf{24} (2017), no.~3, 321--353.

\bibitem{MatsushitaJCTB}
T.~Matsushita, Hom complexes of graphs whose codomains are square-free, J. Combin. Theory Ser. B \textbf{178} (2026), 267--293.

\bibitem{MW} T.~Matsushita, S.~Wakatsuki, Dominance complexes, neighborhood complexes and combinatorial Alexander duals, J. Combin. Theory Ser. A {\bf 211} (2025) 105978.

\bibitem{Nadolski}
A.~Nadolski, Circular coloring of graphs, in \emph{Graph Colorings} (M.~Kubale, ed.), Contemporary Mathematics, vol.~352, American Mathematical Society, Providence, RI, 2004, pp.~123--137.

\bibitem{Payan}
C.~Payan, On the chromatic number of cube-like graphs, Discrete Math. \textbf{103} (1992), 271--277.

\bibitem{Vince}
A.~Vince, Star chromatic number, J. Graph Theory \textbf{12} (1988), no.~4, 551--559.

\bibitem{WrochnaJCTB}
M.~Wrochna, Square-free graphs are multiplicative, J. Combin. Theory Ser. B \textbf{122} (2017), 479--507.

\bibitem{ZhuSurvey}
X.~Zhu, Circular chromatic number: A survey, Discrete Math. \textbf{229} (2001), 371--410.

\bibitem{ZhuRecent}
X.~Zhu, Recent developments in circular colouring of graphs, in \emph{Topics in Discrete Mathematics} (M.~Klazar, J.~Kratochv\'{i}l, M.~Loebl, J.~Matou\v{s}ek, R.~Thomas, and P.~Valtr, eds.), Algorithms and Combinatorics, vol.~26, Springer, Berlin, 2006, pp.~497--550.

\end{thebibliography}
\end{document}